\documentclass{amsart}
\usepackage{amssymb,latexsym}
\usepackage{tikz}
\usetikzlibrary{positioning}
\theoremstyle{plain}
\newtheorem{theorem}{Theorem}
\newtheorem{corollary}{Corollary}
\newtheorem*{main}{A~Main~Theorem}
\newtheorem*{mainl}{A~Main~Lemma}
\newtheorem*{mainc}{A~Main~Corollary}
\newtheorem*{main1}{A~Main~Theorem1}
\newtheorem*{mainl1}{A~Main~Lemma1}
\newtheorem*{mainc1}{A~Main~Corollary1}
\newtheorem*{main2}{A~Main~Theorem2}

\newtheorem{lemma}{Lemma}

\theoremstyle{definition}
\newtheorem{definition}{Definition}

\theoremstyle{remark}

\newtheorem*{note}{Note}

\numberwithin{equation}{section}

\begin{document}
\title[NSBDL]{The Nearly Boolean Nature of Core Regular Double Stone Algebras, CRDSA. – \\ \tiny(Ternary set partitions, CRDSA, embeddings and dual equivalences)}
\author{Daniel J. Clouse\\Research\\
        U.S. Department of Defense\\
        Ft. Meade, MD 20755, USA}

\begin{abstract}
In \cite{CRDSA} many useful results regarding the center of a core regular double Stone algebra, CRDSA, are shown that begin to indicate the \textit{nearly Boolean} nature of CRDSA which we focus on here. We start with a model of network security where individual nodes are considered to be in one of 3 states. Let J be any non-empty set of network nodes, not necessarily finite. We define the node set bounded distributive lattice through the pairwise disjoint subsets of J with the well known binary operations of ternary set partitions and note J = {1} is our minimal case. We then show the resultant bounded distributive lattice is isomorphic to $C_3^J$ where $C_3$ is the 3 element chain CRDSA. We derive that \textbf{every CRDSA is a subdirect product of} $\mathbf{C_3}$, similarly as for Boolean algebras and $C_2$. We use these results along with a few known results to show a main result, namely \textbf{every Boolean algebra is the center of some CRDSA}. We then use that result to characterize all subalgebras of a finite core regular double Stone Algebras, namely \textbf{a CRDSA A is a subalgebra of} $\mathbf{C_3^J}$ \textbf{for some finite J} $\mathbf{\leftrightarrow A \cong C_3^K}$ \textbf{for some} $\mathbf{K \leq J}$.\\ \indent Next we show that $\mathbf{C_3}$ \textbf{is primal} which implies that \textbf{the variety generated by} $\mathbf{C_3}$ \textbf{is dually equivalent to the category of Stone spaces and hence the category of Boolean algebras.} In some sense this is a last step towards our goal of establishing CRDSA as \textit{nearly Boolean}, but leaves us a bit dissatisfied as to our understanding of CRDSA in the dual topological category. Hence we continue by establishing a duality between the category of CRDSA and specifically crafted bi-topological spaces that enables better understanding of the "nearly boolean" nature of CRDSA in the dual category.  More succinctly, duality through  refinement of a pre-established duality of pairwise Stone spaces and bounded distributive lattices \cite{guram}, and we use \cite{guram} as our primary source for bi-topological spaces as well.\\ \indent Towards this end we first establish \textbf{necessary and sufficient conditions on a pairwise zero-dimensional space such that it will have a core regular double Stone algebra base}. For the purposes of this note only we will call any pairwise zero-dimensional space $(X,t_1,t_2)$ with a core regular double Stone algebra bases $(B_1,\vee,\wedge,\emptyset,X)$ and $(B_2,\vee,\wedge,\emptyset,X)$ a core regular double pairwise zero-dimensional space. We note that these conditions are indicative of the \textit{nearly Boolean} CRDSA are.(Recall given any topological space X, the collection of subsets of X that are clopen (both closed and open) is a Boolean algebra.) For example, if $u \in B_1$ is not $t_1$ clopen/complemented then $Cl_1(u)\in B_1$ and is $t_1$ clopen/complemented. \\ \indent Next we show that \textbf{the pairwise Stone space derived from a CRDSA L has a base that is a CRDSA isomorphic to L}. For the purposes of this note only we will call any such pairwise Stone space a core regular double pairwise Stone space. Then we establish \textbf{necessary and sufficient conditions for a bi-continuous map to have an inverse that is a CRDSA homomorphism}. These results are topologically indicative of just how ``nearly Boolean'' CRDSA are, these inverses must respect the appropriate conditions on the boundary of non-clopen elements of the bases of $t_1$ and $t_2$. \\ \indent From that result we validate the claim that \textbf{the category of core regular double Stone algebras is dually equivalent to what we call the category of core regular double pairwise Stone spaces}. We note that the conditions for this duality should be able to be easily relaxed to yield a duality for a less rigid subclass of bounded distributive lattices than CRDSA, bounded distributive pseudo-complemented lattices for example.
\end{abstract}

\address{Research\\
        U.S. Department of Defense\\
        Ft. Meade, MD 20755, USA} 
\email{djclouse@gmail.com}
\thanks{*Our thanks to Eran Crockett for noting that $C_3$ is primal, hence Hu's theorem applies, and to Guram Bezhanishvili for very useful correspondence} 
\date{January 01, 2018}

\maketitle
\newpage
\section{Introduction}\label{S:intro} 
We start by proving the following result:

\begin{main} 
Every Boolean algebra is the center of some core regular double Stone algebra.
\end{main}

We use that to derive the following:

\begin{mainc}\label{C3JsubiffC3K}
A CRDSA A is a subalgebra of $C_3^J$ for some finite J $\leftrightarrow A \cong C_3^K$ for some $K \leq J$.
\end{mainc}

We further show

\begin{mainl}\label{c3dual}
The category of CRDSA is dually equivalent to the category of Stone spaces and hence to the category of Boolean algebras. Furthermore, the category of CRDSA is dually equivalent to $V(Z_3)$ where denote the ring of integers modulo 3.
\end{mainl}

This leaves us a bit dissatisfied as to our understanding of CRDSA in the dual topological category. Hence we continue by establishing a duality between the category of CRDSA and specifically crafted bi-topological spaces that enables better understanding of the "nearly boolean" nature of CRDSA in the dual category. We prove the following:

\begin{main1} 
A pairwise zero-dimensional space $(X,t_1,t_2)$ will have a core regular double Stone algebra base $(B_1,\cup,\cap,\emptyset,X) \leftrightarrow$
  \begin{enumerate}
  \item $v \in B_2 \rightarrow Int_1(v) \in B_1$, affording $u^* = Int_1(u^c) = Cl_1(u)^c \in B_1$, the pseudocomplement.
  \item $u \in B_1 \rightarrow Int_2(u) \in B_2$, affording $u^+ = Int_2(u)^c = Cl_2(u^c) \in B_1$, the dual pseudocomplement.
  \item $u \in B_1 \rightarrow Cl_1(u)^c = u^*$ is clopen in $t_1$, affording the Stone identity $u^* \wedge u^{**} = X$ and $u^*$ being complemented.
  \item $u \in B_1 \rightarrow Cl_2(u^c)^c = Int_2(u) = (u^+)^c$ is clopen in $t_2$, affording the Stone Identity $u^+ \wedge u^{++} = \emptyset$ and $u^+$ being complemented.
  \item For any $u,w \in B_1$, $Cl_1(u) = Cl_1(w)$ and $Int_2(u) = Int_2(w) \rightarrow u = w$.
  \item $k(B_1)=\{u \in B_1 | Bd_1(u) = u^c$ and $Bd_2(u^c) = u\} \neq \emptyset$.
  \end{enumerate}
  For the purposes of this note we will call any such pairwise zero-dimensional space $(X,t_1,t_2)$ a core regular double pairwise zero-dimensional space and similarly for the corresponding pairwise Stone spaces.
\end{main1}

Let L be a bounded distributive lattice and X = pf(L) be the set of prime filters of L. We define the following::
\begin{definition}
  $\Phi_+: L \rightarrow \mathcal{P}(X)$ by $\Phi_+(a) = \{x \in X | a \in x\}$ and\\$B_+ = \Phi_+[L] = \{\Phi_+(a) | a \in L\}$.
\end{definition}

\begin{mainl1}\label{subgen1}
If L is a core regular double Stone algebra then $(B_+,\cap,\cup,\emptyset,X)$ is also a core regular double Stone algebra isomorphic to L. 
\end{mainl1}

\begin{main2}
Let $(X,t_1,t_2)$ and $(Y,\gamma_1,\gamma_2)$ be a bitopological space with CRDSA bases $B_1$, $B_2$ and $D_1$, $D_2$, respectively, $f:X \rightarrow Y$ be a bi-continuous map and $u,w \in D_1$. The $f^{-1}$ is a CRDSA homomorphism $\leftrightarrow f^{-1}(Bd_1(u)) \subseteq Cl_1(f^{-1}(u))$ and $f^{-1}(Bd_2(u^c)) \subseteq Cl_2(f^{-1}(u^c))$.
\end{main2}

\begin{mainc1}
The category of core regular double Stone algebras is dually equivalent to the category of core regular double pairwise Stone spaces.
\end{mainc1}

\section{The $C_3^J$ construction}\label{C3J} 
For the basic notation in lattice theory and universal algebra, see Burris and Sankappanavar~\cite{BS}.  We start with some
definitions:

\begin{definition}\label{NSBDLdef}
   Let $J$ be a non-empty set of network nodes and let \\ $L = \{(X_1,X_2)| X_1,X_2 \subseteq J $ such that $X_1 \cap X_2 = \emptyset \}$. \\ We define binary operations $\vee$ and $\wedge$ on L as follows:
\begin{itemize}
	\item $(X​_1,X​_2​) \vee ​(Y​_1​,Y​_2​) = (X​_1​​ \cup ​Y​_1​,X​_2​ \cap Y​_2​)$ and;\label{m-i}
	\item $(X​_1,X​_2​) \wedge ​(Y​_1​,Y​_2​) = (X​_1​​ \cap ​Y​_1​,X​_2​ \cup Y​_2​)$. \label{j-i}
	\begin{itemize}
		\item The fact that $L$ is a bounded distributive lattice with bounds $(J,\emptyset)$ and $(\emptyset,J)$ is well known.\label{c-s}
	\end{itemize}
   \end{itemize}
We refer to the operations $\vee$, $\wedge$ as join, meet respectively, $X​_1$ are the "known good" nodes, $X​_2$ ​ are the "known bad" and the nodes in $(X​_1​ \cup X​_2​)^c$ are unknown. We call this lattice the "Node Set bounded distributive lattice" and denote it $NS_J$ for a given node set $J$.
\end{definition}

We note that if $J$ is defined and $X_1$ and $X_2$ are understood, we make no mention of $(X​_1​ \cup X​_2​)^c$. We further note that in addition to $(J,\emptyset)$ and $(\emptyset,J)$, there is another very important element of $NS_J$ and it is $S=(\emptyset,\emptyset)$. Lastly, it is very important to note that these binary operations coincide with the following partial ordering on L where ​$\leq$​ denotes the familiar “less than or equal to” concept:

\begin{itemize}
	\item $(X​_1,X​_2​) \leq ​(Y​_1​,Y​_2​) \leftrightarrow X​_1​​ \subseteq ​Y​_1$ and $Y​_2​​ \subseteq ​X_2$
\end{itemize}

So {\bf“moving up” = “more known good” nodes and/or “less known bad” nodes} and conversely! Our minimal example, $J=\{1\}$ is evidently the 3 element chain denoted $C_3$ and Figure 1. gives the Hasse diagrams of $C_3$ and also $C_3^2$ for $J=\{1,2\}$. We find these diagrams useful, although we mention $J=\{1,2,3\}$ is the minumum case where at least one node can be in all 3 states $X​_1$, $X​_2$ ​and $(X​_1​ \cup X​_2​)^c$.

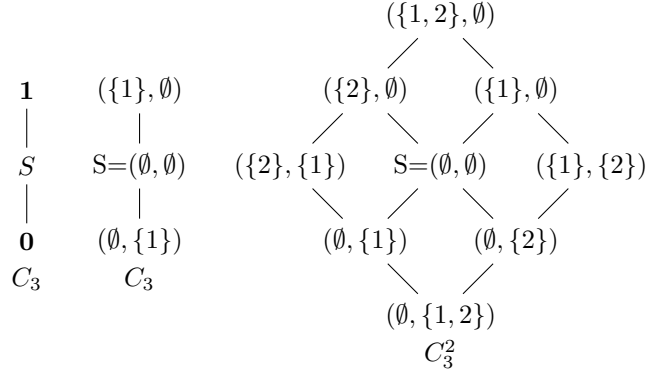
\begin{figure}\label{c3}
\caption{Minimal Examples, $C_3$ and $C_3^2$}
    \begin{center}
    \begin{tikzpicture}
        \tikzstyle{every node} = [rectangle]
	\node (s0) at (0,.5) {$C_3^2$};
        \node (s1) at (0,1) {($\emptyset,\{1,2\}$)};
        \node (s2) at (-1,2) {($\emptyset,\{1\}$)};
        \node (s3) at (1,2) {($\emptyset,\{2\}$)};
        \node (s4) at (-2,3) {($\{2\},\{1\}$)};
	\node (s5) at (0,3) {S=($\emptyset,\emptyset$)};
	\node (s6) at (2,3) {($\{1\},\{2\}$)};     
        \node (s7) at (-1,4) {($\{2\},\emptyset$)};
	\node (s8) at (1,4) {($\{1\},\emptyset$)};
	\node (s9) at (0,5) {($\{1,2\},\emptyset$)};
	\node (s13) at (-4,1.5) {$C_3$};
	\node (s10) at (-4,2) {($\emptyset,\{1\}$)};
	\node (s11) at (-4,3) {S=($\emptyset,\emptyset$)};
	\node (s12) at (-4,4) {($\{1\},\emptyset$)};
	\node (s17) at (-5.5,1.5) {$C_3$};
	\node (s14) at (-5.5,2) {$\mathbf{0}$};
	\node (s15) at (-5.5,3) {$S$};
	\node (s16) at (-5.5,4) {$\mathbf{1}$};
        \foreach \from/\to in {s1/s2, s1/s3, s2/s4, s2/s5, s3/s5, s3/s6, s4/s7, s5/s7, s5/s8, s6/s8, s7/s9, s8/s9, s10/s11, s11/s12, s14/s15, s15/s16}
            \draw[-] (\from) -- (\to);
    \end{tikzpicture}
    \end{center}
\end{figure}

\begin{definition}\label{2p}
An element $x$ of a bounded lattice $L$ ​is said to have a pseudocomplement if there exists a greatest element​ $x^* \in L$, disjoint from x, with the property that $x \wedge x^* = 0$. An element $x \in l$ is said to have a dual pseudocomplement if there exists a ​least element​ $x^+ \in L$, disjoint from x, with the property that $x \vee x^+ =1$. The lattice L itself is called a ​doubly​ ​pseudocomplemented lattice​ if every element of L is simultaneously pseudocomplemented and dually pseudocomplemented.
\end{definition}

It is known that every finite distributive lattice is pseudocomplemented. We call out unary operations $^*$ and $^+$ for $NS_J$ here and claim without proof that under them $NS_J$ is a doubly pseudocomplemented lattice for arbitrary $J$.

\begin{definition}\label{NSBDL2p}
Let $(X_1,X_2) \in NS_J$ 
\begin{itemize}
	\item $(X_1,X_2)^*=(X_2,X_2^c)$ and
	\item $(X_1,X_2)^+=(X_1^c,X_1)$
\end{itemize}
define the psuedocomplement and dual pseudocomplement operations of $NS_J$.
\end{definition}
We will now treat $NS_J$ as doubly pseudocomplemented prove the following result:

\begin{lemma}\label{L:ds} 
   $NS_J \cong C_3^J$
\end{lemma}

\begin{proof} 
   We define $\alpha : NS_J \rightarrow C_3^J$ pointwise for $(X_1,X_2) \in NS_J$ by $\alpha(X_1,X_2)=(x_i)_{i \in J}=$
\begin{enumerate}
	\item 1 if $i \in X_1$
	\item S if $i \in (X​_1​ \cup X​_2​)^c$
	\item 0 if $i \in X_2$
\end{enumerate}
Clearly $\alpha$ is bijective and to see that $\alpha$ and $\alpha^{-1}$ are order preserving follows from the partial order itself. Recall that $(X​_1,X​_2​) \leq ​(Y​_1​,Y​_2​) \leftrightarrow X​_1​​ \subseteq ​Y​_1$ and $Y​_2​​ \subseteq ​X_2$, we consider $\alpha^{-1}$ and $\alpha$ below:
\begin{enumerate}
	\item $\alpha^{-1}$,suppose $(x_i) \leq (y_i)$
	\begin{enumerate}
		\item $x_i=1 \rightarrow y_i=1$ and hence $X_1 \subseteq Y_1$.
		\item $y_i=0 \rightarrow x_i=0$ and hence $Y_2 \subseteq X_2$.
	\end{enumerate}	
	\item $\alpha$, suppose $(X_1,X_2) \leq (Y_1,Y_2)$ and consider their images $(x_i)$ and $(y_i)$
	\begin{enumerate}
		\item Suppose $x_i=1$, then $X_1 \subseteq Y_1 \rightarrow y_i=1$.
		\item Suppose $x_i=S$, if $y_i=0$ then $\exists i$ such that $i \in Y_2$ and $i \in (X​_1​ \cup X​_2​)^c$ yielding a contradiction with $Y​_2​​ \subseteq ​X_2$.
	\end{enumerate}
\end{enumerate}
The fact that it follows that $\alpha$ is a lattice isomorphism is well known, e.g. Theorem 2.3 of ~\cite{BS}. 
The fact that $\alpha$ preserves $*$ and $+$ follows from their definitions and the fact that they are extended point-wise from $C_3$ to $C_3^J$. Recall: 
\begin{enumerate}
	\item $* : L \rightarrow L$ is defined by $(X_1,X_2) \rightarrow (X_2,X_2^c)$ and;
	\item $+ : L \rightarrow L$ is defined by $(X_1,X_2) \rightarrow (X_1^c,X_1)$ so that 
	\begin{enumerate}
		\item $* : C_3 \rightarrow C_3$ is defined by $0 \rightarrow 1$, $S \rightarrow 0$ and $1 \rightarrow 0$;
		\item $+ : C_3 \rightarrow C_3$ is defined by $0 \rightarrow 1$, $S \rightarrow 1$ and $1 \rightarrow 0$;
	\end{enumerate}	
\end{enumerate}
The fact that $\alpha$ preserves the constants $0 = (\emptyset,J)$ and $1 = (J,\emptyset)$ is clear and we note that it also preserves $S = (\emptyset,\emptyset)$. Lastly, $*$ and $+$ are clearly (dual)psuedocomplements of $C_3$ and hence also for $C_3^J$.
\end{proof}

\begin{note}
As a result of Lemma \ref{L:ds} we could declare all 3 elements of $C_3$ to be constants and consider it as a generating structure for our \textit{universe of node set bounded distributive lattices}, e.g. (sub)direct products of $C_3$.
\end{note}

\begin{definition}\label{DSA}
A double Stone algebra (DSA) $<L,\wedge,\vee,*,+,0,1>$ is an algebra of type $<2,2,1,1,0,0>$ such that:
\begin{enumerate}
	\item $<L,\vee,\wedge,0,1>$ is a bounded distributive lattice
	\item $x^*$ is the pseudocomplement of $x$ i.e $y \leq x^* \leftrightarrow y \wedge x = 0$
	\item $x^+$ is the dual pseudocomplement of $x$ i.e $y \geq x^+ \leftrightarrow y \vee x = 1$
	\item $x^* \vee x^{**}=1$, $x^+ \wedge x^{++}=0$, e.g. the Stone identities.
\end{enumerate}
\end{definition}
Condition 4. is what distinguishes from double p-algebra and conditions 2. and 3. are equivalent to the equations
\begin{itemize}
	\item $x \wedge (x \wedge y)^* = x \wedge y^*$, $x \vee (x \vee y)^+ = x \vee y^+$
	\item $x \wedge 0^* = x$, $x \vee 1^+ = x$
	\item $0^{**} = 0$, $1^{++} = 1$
\end{itemize}
so that DSA is an equational class.
\begin{note}
We have already shown that $NS_J \cong C_3^J$ is doubly pseudocomplemented and the definitions of $*$ and $+$ clearly demonstrate Condition 4. above.
\end{note}
\begin{definition}\label{RDSA} A double Stone algebra L is called regular, if it additionally satisfies
\begin{itemize}
	\item $x \wedge x^+ \leq y \vee y^*$
	\begin{itemize}
		\item this is equivalent to $x^+ = y^+$ and $x^* = y^* \rightarrow x = y$
	\end{itemize}
\end{itemize}
\end{definition}
\begin{note}
Regular means any 2 congruences having a common class are the same.
\end{note}
\begin{corollary}
$NS_J \cong C_3^J$ is a regular double Stone Algebra
\end{corollary}
\begin{proof}
We have already shown $C_3^J$ to be a double Stone algebra. Let $x,y \in C_3^J$ and suppose $x^*=y^*$ and $x^+=y^+$.
\begin{itemize}
	\item $x_i=1 \rightarrow y_i^*=0=y_i^+ \rightarrow y_i=1$
	\item $x_i=S \rightarrow ((y​_i​^+=1 \rightarrow y​_i= 0$ or $S)$ and $(y​_i​^*=0 \rightarrow y_i= S$ or $1))​ \rightarrow​ y​_i=S$
	\item $x_i=0 \rightarrow y_i^+=1 \rightarrow y_i=0$
\end{itemize}
So, we see that $x = y$ and $C_3^J$ is regular.
\end{proof}
It is straightforward to see that $NS_J$ is regular under the operations from Definition \ref{RDSA}. We now note some useful definitons and known results.
\begin{definition}\label{center}
Let A be a regular double Stone algebra. An element $a \in A$ is called central if $a^*=a^+$. The set of all central elements of A is called the center of A and is denoted by C(A); that is $C(A)=\{a \in A|a^*=a^+\}=\{a^* |a \in a\}=\{a^+ |a \in a\}$.  
\end{definition}

Now we cite the following Theorem from \textit{Centre of Core Regular Double Stone Algebra} \cite{CRDSA}:

\begin{theorem}\label{centerboolsub} 
Let A be a regular double Stone algebra. Then $C(A)$ is a Boolean sub algebra of A with respect to the induced operations $\wedge$, $\vee$ and $^*$.\cite{CRDSA}
\end{theorem}

\begin{definition}\label{dense}
Every element x of a double Stone algebra with the property $x^* = 0$ (or equivalently, $x** =1$) is called dense. Every element of the form $x \vee x^*$ is dense and we denote the set of all dense elements of L, D(L). Every element x with the property $x^+ = 1$ (or equivalently, $x^{++} = 0$) is called dually dense. Every element of the form $x \wedge x^+$ is dually dense and we denote the set of all dual dense elements of L, $\overline{D(L)}$.
\end{definition}

\begin{definition}\label{core}
The core of a double Stone algebra L is defined to be $K(L) = D(L) \cap \overline{D(L)}$ and we call a regular double Stone algebra with non-empty core a core regular double Stone algebra, CRDSA.\cite{CRDSA}
\end{definition}

\begin{note}
We note that $S = (\emptyset,\emptyset) = D(L) \cap \overline{D(L)}$ is the only element of $NS_J$ that is simultaneously dense and dually dense. Clearly $C_3$ and hence $C_3^J$ is a CRDSA and from here forward we treat it as such. For any CRDSA A, $|K(A)| = 1$ follows easily from regularity.(This reflects our desire to define S to be a constant and re-define CRDSA as an algebra of type $<2,2,1,1,0,0,0>$). We recall that $\alpha$ from Lemma \ref{L:ds} preserves $S$ and note that we have shown the following:
\end{note}

\begin{corollary}
$NS_J \cong C_3^J$ as CRDSA.
\end{corollary}

Katrinak showed in \textit{Injective Double Stone Algebras} (1974), that the chains $C_2$, $C_3$, and $C_4$ are the only nontrivial subdirectly irreducible DSAs, then Comer refined for RDSAs in \textit{On Connections Between Information Systems, Rough Sets and Algebraic Logic} (1993).

\begin{theorem}\label{comer}
Every regular double Stone algebra is a subdirect product of $C_2$ and $C_3$.\cite{comer}
\end{theorem}

The following corollary follows by observing that clearly any CRDSA cannot have 2 as a factor.

\begin{corollary}\label{subdir}
Every CRDSA is a subdirect product of $C_3$.
\end{corollary}
We note that this result lends itself very well to refining Theorems 2.6,2.7,2.8 and very likely Theorem 3.1 of \cite{CRDSA}. We have shown that $NS_J \cong C_3^J$ as CRDSA and hence we can pass between the ternary set partition and 3-element chain product representations. In fact, it is easy to show that the former satisfies De Morgan’s Laws for ~ and * over $\wedge$ and $\vee$.

\begin{main}\label{anyB}
Let $B$ be any Boolean Algebra, then $B = C(A)$ for some CRDSA $A$.
\end{main}
\begin{proof}
It is known that $B$ is a subdirect product of $C_2^J$ J for some J \cite{BS}, hence we first embed B into $C_3^J$ as a Boolean algebra and define 
\begin{itemize}
	\item $<B> = \{x \in A| x^+$ and $x^*$ are in B$\}$. 
\end{itemize}
We claim that $<B>$ is a subalgebra of $C_3^J$ as a CRDSA and $C(<B>) = B$. We first note the following:
\begin{enumerate}
	\item $\{0,S,1\} \subseteq <B>$. In fact for every $x \in B$ there exists a $y \in \uparrow S$ and a $z \in \downarrow S$ with $y^+=z^*=x$, $y^*=0$, and
$z^+=1$.
	\item Let $x,y \in <B>$, recall $(x \vee y)^* = x^* \wedge y^*= a \wedge b$ for some $a,b \in B$ and $<B>$ is closed under $\vee$,
similarly for $\wedge$.
	\item Closure under $^*$ and $^+$ follow by the definition of $<B>$.
	\item $^*$ and $^+$ satisfying the (dual) Stone identities is inherited.
	\item Regularity is inherited.
	\item $K(<B>) = S$.
\end{enumerate}
Hence $<B>$ is a subalgebra of $C_3^J$.\\
Clearly $B \subseteq C(<B>)= \{x \in <B>| x = x^{**}\}$. If $x \in C(<B>)$ then $x \in C_2^J$ and $x^* \in B \rightarrow x \in B$. 
\end{proof}

We now cite 2 useful results from \textit{Centre of Core Regular Double Stone Algebra}, Srikanth, Kumar, we will use in the next section.\cite{CRDSA}

\begin{theorem}\label{coreexp}
If A is CRDSA with core element k, then every element x of A can be written as $x = x^{**} \wedge (x^{++} \vee k)$
and $x = x^{++} \vee (x^{**} \wedge k)$.
\end{theorem}
\begin{theorem}\label{centiso}
Two CRDSAs are isomorphic if and only if their centers are isomorphic.
\end{theorem}

We also cite this well known result of Boolean algebras which we will use in the note at the end of the next section.\cite{givhal}

\begin{corollary}\label{subgen}
Let $B$ be a subalgebra of a Boolean algebra A and $a \in A$ then $<B \cup \{a\}>$ = $\{(b \wedge a) \vee (c \wedge a^*)|b,c \in A\}$.
\end{corollary}

We close this section with the following corollary to  Theorem \ref{coreexp} as we find it interesting as it was part of the motivation for Theorem \ref{anyB}.
\begin{corollary}\label{subgen2}
If A is CRDSA, then $A = <\{C(A) \cup S\}>$ where $<>$ can be defined as closure under $\vee$ and $\wedge$.
\end{corollary}

\section{Some useful results regarding finite CRDSA}\label{C3Jfinite} 
In this section we assume A is a finite CRDSA and state some useful facts.

\begin{lemma}\label{C3JcontainsC3K}
$C_3^J$ contains subalgebras isomorphic to $C_3^K$ for all $K \leq J$.
\end{lemma}
\begin{proof}
$C(C_3^K)$ is isomorphic to $C_2^K$ and it is known that $C_2^K$ can be embedded in $C_2^J$ for all such K and J, \cite{givhal}, now use proof of Theorem \ref{anyB}.
\end{proof}

\begin{mainc}
A CRDSA A is a subalgebra of $C_3^J$ for some finite J $\leftrightarrow A \cong C_3^K$ for some $K \leq J$.
\end{mainc}
\begin{proof}
$\leftarrow$ By Theorem \ref{centerboolsub} $C(A) \cong C_2^K$ and hence consider $C_2^K \hookrightarrow C_2^J \hookrightarrow C_3^J$. Now use the Theorem \ref{anyB} and Theorem \ref{centiso}.\\
$\rightarrow$ This follows directly from Theorem \ref{centerboolsub} and Theorem \ref{centiso}.
\end{proof}
\begin{note}
We note that the corollary above yields an algorithm for finding all the subalgebras of $C_3^J$, for finite $J$, namely find all $B \leq C_2^J$ using Corollary \ref{subgen}, embed them in $C_3^J$ then use Corollary \ref{subgen2} and Theorem \ref{anyB}.
\end{note}

\section{Initial Duality for CRDSA}\label{initdual} 

Now we list some definitions for which we will not need to go into detail as we will only leverage theorems that require them. 
\begin{definition}\label{pdefs}We refer the interested reader to \cite{bergman} for more details.
  \begin{enumerate}
    \item An algebra A is primal if for every $n \geq 1$ and every operation $f: A^n \rightarrow A$, there is a term operation $t: A^n \rightarrow A$ such that $t(a)=f(a)$ for all $a \in A^n$.
    \item An algebra A is rigid if it has no non-identity automorphisms.
    \item An algebra A is simple if it only has $\Delta$ and $\nabla$ as congruences.
    \item A variety V is arithmetical if it is both congruence-distributive and congruence-permutable.
  \end{enumerate}
\end{definition}
We will use the two Theorems below for this section's Main Lemma (Hu 1971):
\begin{theorem}\label{primal}\cite{bergman} Let A be a finite algebra. Then A is primal $\leftrightarrow$
  \begin{enumerate}
    \item A has no proper subalgebras,
    \item A is simple
    \item A is rigid, and
    \item the variety generated by A, V(A), is arithmetical.
  \end{enumerate}
\end{theorem}
\begin{theorem}\label{primdual}
Let P be a finite primal algebra. Then $V(P)$ is dually equivalent to the category of totally-disconnected compact Hausdorff topological spaces.\cite{bergman}
\end{theorem}

\begin{note}
  \indent Please recall the Boolean algebras are definitionally equivalent to Boolean Rings\cite{givhal}. Let $Z_3$ denote the ring of integers modulo 3 and note it satisfies $x^3=x$.
\end{note}

\begin{lemma}
$Z_3$ is primal.
\end{lemma}
\begin{proof}
  Let $Z_3=\{0,1,2\}$ and consider the following:
  \begin{enumerate}
  \item $1 + 1 = 2$ and 0,1 are both constants so $Z_3$ has no proper sub-algebras.
  \item To see that $Z_3$ is simple suppose $\Theta \in Con(Z_3)$ and view the following cases:
    \begin{enumerate}
    \item $0$ $\Theta$ $1 \rightarrow (0 * 2)$ $\Theta$ $(1 * 2) \rightarrow  0$ $\Theta$ $2 \rightarrow$ $\Theta$ = $\nabla$.
    \item $0$ $\Theta$ $2 \rightarrow (0 + 2)$ $\Theta$ $(2 + 2) \rightarrow  2$ $\Theta$ $1 \rightarrow$ $\Theta$ = $\nabla$.
    \item $1$ $\Theta$ $2 \rightarrow (1 + 2)$ $\Theta$ $(2 + 2) \rightarrow  0$ $\Theta$ $1 \rightarrow$ $\Theta$ = $\nabla$.
    \end{enumerate}
  \item $Z_3$ has no non-identity automorphisms as 0 and 1 are constants, hence it is rigid.
  \item As $Z_3$ satisfies $x^3=x$ it follows that $m(x,y,z)=z - (z - y)(z - x)^2$ is a majority term and hence $V(Z_3)$ is congruence distributive. 
  \item $p(x,y,z) = x - y + z$ is a Mal’cev term for rings and hence $V(Z_3)$ is congruence permutable.
\end{enumerate}
\end{proof}

\begin{lemma}\label{c3primal}
$C_3$ is primal.
\end{lemma}
\begin{proof}
\begin{enumerate}
  \item That $C_3$ has no proper subalgebras and is rigid is clear.
  \item That $C_3$ is simple follows by a case argument considering elements outside of $\Delta$. Let $\theta \in Con(C_3)$, $0$ $\theta$ $1 \rightarrow (S \vee 0)$ $\theta$ $(S \vee 1) \rightarrow$ $\theta=\nabla$. The cases $0$ $\theta$ $S$ and $S$ $\theta$ $1$ follow from $0$ $\theta$ $1$ utilizing $^*$ and $^+$, respectively.
  \item Since $C_3$ is lattice, $V(C_3)$ is congruence-distributive and that it is congruence-permutable follows from the fact that a mal'cev term exists:\\$m(x,y,z)=((x \vee z) \wedge y^*) \vee ((x \vee z^{++}) \wedge ((x \wedge z) \vee y^+) \wedge z^{**})$.\cite{ec}
\end{enumerate}
\end{proof}

\begin{mainl}
The category of CRDSA is dually equivalent to the category of Stone spaces and hence to the category of Boolean algebras. Furthermore, the category of CRDSA is dually equivalent to $V(Z_3)$ where denote the ring of integers modulo 3.
\end{mainl}
\begin{proof}
The result follows directly from Lemma \ref{c3primal} and Theorem \ref{primdual}
\end{proof}

\indent These results are an important part of our goal of establishing CRDSA as \textit{nearly Boolean}, but leave us a bit dissatisfied as to our understanding of the dual topological category. As we continue we will establish a duality between the category of CRDSA and bi-topological spaces that enables better understanding of the "nearly boolean" nature of CRDSA in the dual category and how the categorical equivalence is enforced.

\section{Facts From \textit{Bitopological duality for distributive lattices and Heyting algebras}}\label{gfacts} \cite{guram}
Recall that a bitopological space is a triple $(X,t_1,t_2)$, where X is a (non-empty) set and $t_1$ and $t_2$ are two topologies on X. For a bitopological space $(X,t_1,t_2)$, let $\delta_1$ denote the collection of closed subsets of $(X,t_1)$ and $\delta_2$ denote the collection of closed subsets of $(X,t_2)$. We start with some definitions:

\begin{definition}\label{PT2}
Let $(X,t_1,t_2)$ be a bitopological space. We say $(X,t_1,t_2)$ is pairwise T2 or pairwise Hausdorff if for any two distinct points $x,y \in X$, there exist disjoint $U \in t_1$ and $V \in t_2$ such that $x \in U$ and $y \in V$ or there exist disjoint $U \in t_2$ and $V \in t_1$ with the same property.
\end{definition}

\begin{definition}\label{P0D}
We say a bitopological space $(X,t_1,t_2)$ is pairwise zero-dimensional if opens in $(X,t_1)$ closed in $(X,t_2)$ form a basis for $(X,t_1)$ and opens in $(X,t_2)$ closed in $(X,t_1)$ form a basis for $(X,t_2)$; that is, $B_1 = t_1 \cap \delta_2$ is a basis for $t_1$ and $B_2 = t_2 \cap \delta_1$ is a basis for $t_2$.
\end{definition}

A useful lemma.
\begin{lemma}\label{zerodimequ} 
Suppose $(X,t_1,t_2)$ is pairwise zero-dimensional. Then the following conditions are equivalent:
\begin{enumerate}
\item $(X,t_1)$ is T0.
\item $(X,t_2)$ is T0.
\item $(X,t_1,t_2)$ is pairwise T2.
\item For any two distinct points $x,y \in X$, there exist disjoint $U,V \in t_1 \cup t_2$ such that $x \in U$ and $y \in V$.
\item $(X,t_1,t_2)$ is join T2.
\item $(X,t_1,t_2)$ is bi-T0.
\end{enumerate}
For an expostion on 5. and 6. above we refer the reader to \textit{Bitopological duality for distributive lattices and Heyting algebras}
\end{lemma}

\begin{definition}\label{PC}
We say a bitopological space $(X,t_1,t_2)$ is pairwise compact if for each cover $\{u_i | i \in I\}$ of X with $u_i \in t_1 \cup t_2$, there exists a finite subcover. We denote the collections of compact subsets of $(X,t_1,t_2)$  $\sigma _1$ and $\sigma _2$. 
\end{definition}

Another useful lemma.
\begin{lemma}\label{paircomp} 
A bitopological space $(X,t_1,t_2)$ is pairwise compact if and only if $\delta _1 \subseteq \sigma _2$ and $\delta _2 \subseteq \sigma _1$.
\end{lemma}

\begin{definition}\label{PSS}
We say $(X,t_1,t_2)$ is a pairwise Stone space if it is pairwise compact, pairwise Hausdorff and pairwise zero-dimensional.
\end{definition}

We will use PStone to denote the category of pairwise Stone spaces and bi-continuous maps: that is, maps that are continuous with respect to both topologies, and DLat to denote the category of bounded distributive lattices.\\

Let L be a bounded distributive lattice and X = pf(L) be the set of prime filters of L. We define the following::
\begin{definition}\label{phis}
  $\Phi_+, \Phi_− : L \rightarrow \mathcal{P}(X)$ by
  \begin{enumerate}
    \item $\Phi_+(a) = \{x \in X | a \in x\}$ and
    \item $\Phi_-(a) = \{x \in X | a \notin x\}$.
  \end{enumerate}
Let $B_+ = \Phi_+[L] = \{\Phi_+(a) | a \in L\}$, $B_- = \Phi_-[L] = \{\Phi_-(a) | a \in L\}$, $t_+$ be the topology generated by $B_+$, and $t_-$ be the topology generated by $B_-$.
\end{definition}

\begin{note}
We wish to note that the fact that $\Phi_+$ is a bounded lattice isomorphism from L to  $(B_+,\cap,\cup,\emptyset,X)$ is established in \cite{guram}. 
\end{note}

\begin{lemma}
$(X,t_+, t_-)$ is a pairwise Stone space. 
\end{lemma}

\begin{note}
If $(X,t_1,t_2)$ is pairwise zero-dimensional, then $B_2 = \{u^c| u \in B_1\}$ and $B_1 = \{v^c | v \in B_2\}$. Moreover, both $B_1$ and $B_2$ contain $\emptyset$ and X, are closed with respect to finite unions/intersections, form bounded distributive lattices and $(B_1,\cap,\cup,\emptyset,X)$  is dually isomorphic to $(B_2,\cup,\cap,X,\emptyset)$.
\end{note}

We have established enough basic facts as to the relationship between Dlat and PStone to give the reader some confidence that the following is true \cite{guram}:

\begin{theorem}\label{DlattoPSS}
The functors outlined above establish a dual equivalence between DLat and PStone.
\end{theorem}
\begin{proof}
  We refer the reader to \textit{Bitopological duality for distributive lattices and Heyting algebras}.\cite{guram}\\
\end{proof}

\section{Enforcing more structure on $(B_1,\cup,\cap,\emptyset,X)$}\label{B1struct}

If $(X,t_1,t_2)$ is a pairwise Stone space, we wish to place necessary and sufficient restrictions on it such that $(B_1,\cup,\cap,\emptyset,X)$ is a CRDSA. Let Int denote interior of, Cl denote closure of and Bd denote boundary of topological spaces. First we look at conditions for (dual) pseudocomplementation, $^*$ and $^+$. 

\begin{lemma}\label{pcdpc}
Let $(X,t_1,t_2)$ be a pairwise zero-dimensional space, then \\$u \in (B_1,\cup,\cap,\emptyset,X)$ has a dual pseudocomplement $w \leftrightarrow w^c$ is the pseudocomplement of $u^c$ in the dual lattice $(B_2,\cup,\cap,X,\emptyset)$ and dually for a pseudocomplement $w$.
\end{lemma}
\begin{proof}
$w \in B_1$ is minimal with respect to $u \cup w = X \leftrightarrow w^c \in B_2$ is maximal with respect to $u^c \cap w^c = \emptyset$ and dually if $w$ is a pseudocomplement. 
\end{proof}

We note that Lemma \ref{pcdpc} has a more general statement for dually isomorphic lattices.

\begin{lemma}Let $(X,t_1,t_2)$ be a pairwise zero-dimensional space, then $(B_1,\cup,\cap,\emptyset,X)$ is a double p-algebra $\leftrightarrow (X,t_1,t_2)$ satisfies the conditions below:
  \begin{enumerate}
  \item $v \in B_2 \rightarrow Int_1(v) \in B_1$, then $(B_1,\cap,\cup,\emptyset,X)$ has a pseudocomplement operation $^*$, namely $u^* = Int_1(u^c) = Cl_1(u)^c$ for all $u \in B_1$.
  \item $u \in B_1 \rightarrow Int_2(u) \in B_2$, then $(B_1,\cap,\cup,\emptyset,X)$ has a dual pseudocomplement operation $^+$, namely $u^+ =  Int_2(u)^c = Cl_2(u^c)$ for all $u \in B_1$. 
  \end{enumerate}
\end{lemma}
\begin{proof}
Recall $B_1 = \{v^c | v \in B_2\} = t_1 \cap \delta _2$ and $B_2 = \{u^c| u \in B_1\} = t_2 \cap \delta _1$. 
\begin{enumerate}
  \item $\rightarrow$ For all $u \in B_1$, $u = v^c$ for some $v \in B_2$, therefore $u^c = v$ and $Int_1(v)$ is the largest open set contained in v and must be in $B_1$ by supposition. The fact that $^*$ is the psuedocomplement follows.\\
$\leftarrow$ Now suppose that $B_1$ has a pseudocomplement $^*$ and recall $Int_1(u^c) \in t_1$ is maximal with respect to $u \cap Int_1(u^c) = \emptyset$ and hence $u^* \subseteq Int_1(u^c)$. If $Int_1(u^c) = \emptyset$ we have equality so suppose $x \in Int_1(u^c)$ such that $x \notin u^c$. As $B_1$ is a basis for $t_1$, $Int_1(u^c) = \cup b_i$ for some $\{b_i | b_i \in B_1\}$ and hence $x \in b_i$ for some i. Now consider $u \cup b_i \in B_1$ contradicting the maximality of $u^*$. 
  \item $\rightarrow$ For all $v \in B_2$, $v = u^c$ for some $u \in B_1$, therefore $v^c = u$ and $Int_2(u)$ is the largest open set contained in u, hence $Int_2(u)^c$ is minimal closed in $t_2$ containing $u^c$ and is in $B_1$ by the supposition $Int_2(u) \in B_2$. That fact that $^+$ is the dual psuedocomplement follows.\\
$\leftarrow$ Now suppose that $B_1$ has a dual pseudocomplement $^+$ and recall $Int_2(u)^c$ is minimal closed in $t_2$ containing $u^c$, hence $B_1 =  t_1 \cap \delta _2 \rightarrow Int_2(u)^c \subseteq u^+$. If $u^+ = \emptyset$ we have equality, so suppose $x \in u^+$ such that $x \notin Int_2(u)^c$. Now we consider $(u^+)^c \subseteq  Int_2(u)$  and using Lemma \ref{pcdpc} the proof follows as in (1).
\end{enumerate}
\end{proof}
Now we look at extending $B_1$ as a double p-algebra to a double Stone algebra, e.g. $^*$ and $^+$ satisfy the Stone identies $x^* \vee x^{**} = X$ and $x^+ \wedge x^{++} = \emptyset$.
\begin{lemma}If $(X,t_1,t_2)$ is a pairwise zero-dimensional space such that $(B_1,\cup,\cap,\emptyset,X)$ is a double p-algebra then $(B_1,\cup,\cap,\emptyset,X)$ is a double Stone algebra $\leftrightarrow (X,t_1,t_2)$ satisfies the conditions below:
\begin{enumerate}
  \item $u^* = Int_1(u^c) = Cl_1(u)^c$ is $t_1$ clopen.
  \item ($u^+)^c = (Cl_2(u^c))^c = Int_2(u)$ is $t_2$ clopen.
\end{enumerate}
\end{lemma}
\begin{proof}
Suppose $u \in B_1$.
\begin{enumerate}
\item $Cl_1(u)^c$ is $t_1$ clopen $\leftrightarrow u^* \vee u^{**} = Cl_1(u)^c \cup Cl_1(Cl_1(u)^c)^c = X$.
  \begin{enumerate}
    \item $\rightarrow Cl_1(u)^c \subseteq Cl_1(Cl_1(u)^c)$ with equality if $Cl_1(u)^c = u* = Int_1(u^c)$ is clopen in $B_1$, hence $u^* \vee u^{**} = Cl_1(u)^c \cup Cl_1(Cl_1(u)^c)^c = Cl_1(u)^c \cup Cl_1(u)$.\\$\leftarrow$ Now suppose $u^* \vee u^{**} = Cl_1(u)^c \cup Cl_1(Cl_1(u)^c)^c = X \rightarrow Cl_1(u) \subseteq Cl_1(Cl_1(u)^c)^c \rightarrow Cl_1(Cl_1(u)^c) \subseteq Cl_1(u)^c \rightarrow Cl_1(Cl_1(u)^c) = Cl_1(u)^c$ and hence $Cl_1(u)^c$ is $t_1$ clopen.
  \end{enumerate}
    \item $Cl_2(u^c)^c$ is $t_2$ clopen $\leftrightarrow u^+ \wedge u^{++} = Cl_2(u^c) \cap Cl_2(Cl_2(u^c)^c) = \emptyset$.
  \begin{enumerate}
    \item $\rightarrow Cl_2(u^c)^c \subseteq Cl_2(Cl_2(u^c)^c)$ with equality if $Cl_2(u^c)^c = (u^+)^c = (Int_2(u)^c)^c = Int_2(u)$ is $t_2$ clopen, hence $u^+ \wedge u^{++} = Cl_2(u^c) \cap Cl_2(u^c)^c$.\\$\leftarrow$ Now suppose $x^+ \wedge x^{++} = Cl_2(u^c) \cap Cl_2(Cl_2(u^c)^c) = \emptyset \rightarrow Cl_2(Cl_2(u^c)^c) \subseteq Cl_2(u^c)^c \rightarrow Cl_2(u^c)^c = Cl_2(Cl_2(u^c)^c)$ and hence $Cl_2(u^c)^c$ is $t_2$ clopen.
  \end{enumerate}
\end{enumerate}
\end{proof}
\begin{note} If $(B_1,\cup,\cap,\emptyset,X)$ is a double Stone algebra the following hold:
  \begin{enumerate}
    \item $\{u^* | u \in B_1\} = \{u^+ | u \in B_1\}$ \cite{duns}
    \item $u^*$ is $t_1$ clopen, $u^*$ and $u^{**}$ are complementary.
    \item $(u^+)^c$ is $t_2$ clopen, $u^+$ and $u^{++}$ are complementary.
    \item $u^{**}=Cl_1(u) \in B_1$ and is $t_1$ clopen/complemented.
  \end{enumerate}
\end{note}
Next we will examine the concepts of density in $B_1$ as a double Stone algebra.
\begin{lemma}\label{coreprop}
  Let $(B_1,\cup,\cap,\emptyset,X)$ be a double Stone algebra, then
  \begin{enumerate}
    \item $u \in D(B_1) \leftrightarrow Bd_1(u) = u^c$
    \item $u \in \overline{D(B_1)} \leftrightarrow Bd_2(u^c) = u$
  \end{enumerate}
\end{lemma}
\begin{proof} 
Suppose $u \in B_1$.
  \begin{enumerate}
    \item $u^* =  \emptyset = Cl_1(u)^c \leftrightarrow Cl_1(u) = X = Bd_1(u) \cup u \leftrightarrow Bd_1(u) = u^c$.
    \item $u^+ = X = Cl_2(u^c) = Bd_2(u^c) \cup u^c \leftrightarrow Bd_2(u^c) = u$.
  \end{enumerate}
\end{proof}
We note that these elements have the property that there closure is all of $X$ in either $t_1$ or $t_2$. We now examine the core in terms of $t_1$ and $t_2$. 
\begin{corollary}\label{coren}
  Let $(B_1,\cup,\cap,\emptyset,X)$ be a double Stone algebra, then\\$k(B_1)=\{u \in B_1 | Bd_1(u) = u^c$ and $Bd_2(u^c) = u\}$. 
\end{corollary}
\begin{proof}
This result follows directly from Lemma \ref{coreprop} and the definition of $k(B_1)$.
\end{proof}
Now we will examine the condition of regularity.
\begin{note}\label{regn}
Let $(B_1,\cup,\cap,\emptyset,X)$ be a double Stone algebra and $u,w \in B_1$, then $u^*=w^*$ and $u^+=w^+ \rightarrow u=w$ is equivalent to\\$Cl_1(u)^c = Cl_1(w)^c$ and $Cl_2(u^c) = Cl_2(w^c) \rightarrow u = w$.\\We further note this is equivalent to\\$Cl_1(u) = Cl_1(w)$ and $Int_2(u) = Int_2(w) \rightarrow u = w$.
\end{note}
We now have enough to list the necessary and sufficient conditions for a pairwise zero-dimensional space $(X,t_1,t_2)$ to have a core regular double Stone algebra $(B_1,\cup,\cap,\emptyset,X)$.
\begin{main1}
  A pairwise zero-dimensional space $(X,t_1,t_2)$ will have a core regular double Stone algebra $(B_1,\cup,\cap,\emptyset,X) \leftrightarrow$
  \begin{enumerate}
  \item $v \in B_2 \rightarrow Int_1(v) \in B_1$, affording $u^* = Int_1(u^c) = Cl_1(u)^c \in B_1$, the pseudocomplement.
  \item $u \in B_1 \rightarrow Int_2(u) \in B_2$, affording $u^+ = Int_2(u)^c = Cl_2(u^c) \in B_1$, the dual pseudocomplement.
  \item $u \in B_1 \rightarrow Cl_1(u)^c = u^*$ is clopen in $t_1$, affording the Stone identity $u^* \wedge u^{**} = X$ and $u^*$ being complemented.
  \item $u \in B_1 \rightarrow Cl_2(u^c)^c = Int_2(u) = (u^+)^c$ is clopen in $t_2$, affording the Stone Identity $u^+ \wedge u^{++} = \emptyset$ and $u^+$ being complemented.
  \item For any $u,w \in B_1$, $Cl_1(u) = Cl_1(w)$ and $Int_2(u) = Int_2(w) \rightarrow u = w$.
  \item $k(B_1)=\{u \in B_1 | Bd_1(u) = u^c$ and $Bd_2(u^c) = u\} \neq \emptyset$.
  \end{enumerate}
  For the purposes of this note we will call any such pairwise zero-dimensional space $(X,t_1,t_2)$ a core regular double pairwise zero-dimensional space. 
\end{main1}
We note that these conditions are indicative of the \textit{nearly Boolean} CRDSA are.(Recall given any topological space X, the collection of subsets of X that are clopen (both closed and open) is a Boolean algebra.) For example, if $u \in B_1$ is not $t_1$ clopen/complemented then $Cl_1(u)\in B_1$ and is $t_1$ clopen/complemented. We note we have only use the fact that $(X,t_1,t_2)$ is a pairwise zero-dimensional space to enforce the desired structure on $(B_1,\cup,\cap,\emptyset,X)$ and assume the further conditions of a pairwise Stone space may be necessary for morphisms of any dual equivalence. We now will now continue towards showing the category of core regular double Stone algebras is dually equivalent to what we call the category of core regular double pairwise Stone spaces. 

\section{Establishing Duality}\label{estdual} 
In this section we begin our work towards refining the dual equivalence between DLat and PStone to only consider those bounded distributive lattices that are actually core regular double Stone algebras. We begin this section with the following Lemma derived from Theorem \ref{DlattoPSS}:   
 
\begin{mainl1}
If L is a core regular double Stone algebra then $(B_+,\cap,\cup,\emptyset,X)$ is also a core regular double Stone algebra isomorphic to L. 
\end{mainl1}
\begin{proof} In this proof we will designate $\Phi_+$ by $\Phi$ and replace $B_+$ with $B_1$ as $^+$ indicates dual-pseudocomplementation.
  \begin{enumerate}
  \item Let $^*: L \rightarrow L$ be the pseudocomplement operation in L, dually for $^+$, and let $x,y \in L$ and $u,w \in B_1$.
    \begin{enumerate}
    \item First we show $(B_1,\cap,\cup,\emptyset,X)$ is a double p-algebra and $\Phi$ is compatible with $^*$ and $^+$. $\Phi(u \wedge u^*)=\Phi(u) \wedge \Phi(u^*) = \emptyset$ and $\Phi(u^*)$ is minimal and exists $\forall u \in B_1$ by the fact that $\Phi$ is lattice isomorphism. Hence $B_1$ has a pseudocomplement operation and $\Phi(u^*) = \Phi(u)^*$. The dual argument works for $^+$ and hence $(B_1,\cap,\cup,\emptyset,X)$ is a double p-algebra.
    \item The Stone identies follow as $u^* \vee u^{**} = \Phi(u^*) \vee \Phi(u^{**}) = X$ and $u^+ \vee u^{++} = \Phi(u^+) \wedge  \Phi(u^{++}) = \emptyset$.
    \item Regularity also follows easily as $\forall u,w \in B_1$ $u = \Phi(x)$ and $w = \Phi(y)$ for some $x,y \in L$. Suppose  $u^*=w^*$ and $u^+=w^+$, then  $\Phi(x^*) = \Phi(y^*)$ and $\Phi(x^+) = \Phi(y^+) \rightarrow x = y \rightarrow u = w$ by the regularity of L and the fact that $\Phi$ is lattice isomorphism.
    \item If $k \in k(L)$ then $\Phi(k^*) = \emptyset$ and $\Phi(k^+) = X$ and hence $\Phi(k) \in k(B_1)$. 
    \end{enumerate}
  \end{enumerate}
  The above shows that $\Phi$ is compatible with $^*$ and $^+$, $(B_1,\cap,\cup,\emptyset,X)$ is a CRDSA, and $\Phi$ preserves core, hence $\Phi$ is a CRDSA isomorphism. 
\end{proof}

To help motivate the need for refining bi-continuous maps we recall the following definition and result:

\begin{definition}
Let A be a regular double Stone algebra. An element $a \in A$ is called central if $a^*=a^+$. The set of all central elements of A is called the center of A and is denoted by C(A); that is $C(A)=\{a \in A|a^*=a^+\}=\{a^* |a \in a\}=\{a^+ |a \in a\}$.  
\end{definition}

\begin{theorem}
  Let A be a regular double Stone algebra, then C(A) is a Boolean sub-algebra of A with respect to the induced operations $\wedge$, $\vee$, and $^*$.\cite{CRDSA}
\end{theorem}

\begin{note}
Let $(Y,\gamma_1,\gamma_2)$ be a bitopological space with bases $D_1$ and $D_2$, respectively. In \textit{Bitopological duality for distributive lattices and Heyting algebras}, it is shown that if $f:X \rightarrow Y$ is a bi-continuous map then $f^{-1}:D_1 \rightarrow B_1$ is a bounded lattice homomorphism. We further note that $C(B_1)$ is exactly the set of clopen subsets of $B_1$ and for a topological space $Y$ we denote the set of clopen subsets of $Y$ by $Clo(Y)$.
\end{note}
   The following results are easily shown:
\begin{lemma}
Let $f:M \rightarrow N$ be a lattice homorphism between two Boolean algebras $M$ and $N$ then the following are equivalent:
\begin{enumerate}
  \item $f(0)=0$ and $f(1)=1$
  \item $f(a')=f(a)'$ for all $a \in M$
\end{enumerate}
\end{lemma}
\begin{proof} We simply utilize the fact that $f$ is compatible with $\wedge$ and $\vee$.
\begin{enumerate}
	\item Suppose $f(0)=0$ then $f(a \wedge a')=f(0)=0$, the other result follows similarly.
	\item Suppose $f(a')=f(a)'$ then $f(0)=f(a \wedge a')=0$, the other result follows similarly.
\end{enumerate}
\end{proof}
Using the Lemma above and the fact that inverse images respect union and intersection we derive a well known result from the Stone duality:
\begin{lemma}
let $f:X \rightarrow Y$ be a continous map, then $f^{-1}|_{Clo(X)}:Clo(X) \rightarrow Clo(Y)$ is a Boolean algebra homomorphism. 
\end{lemma}
So, utilizing Theorem \ref{centerboolsub} without much work we see the following holds:
\begin{corollary}
$f^{-1}|_{C(D_1)}:C(D_1) \rightarrow C(B_1)$ is a Boolean algebra homomorphism.
\end{corollary}

However, that does not deal with compatibility of $f^{-1}$ with $^*$ and $^+$ for $d\notin C(D_1)$. For that we need the following:

\begin{lemma}\label{inv2phom}
Let $(X,t_1,t_2)$ and $(Y,\gamma_1,\gamma_2)$ be a bitopological space with CRDSA bases $B_1$, $B_2$ and $D_1$, $D_2$, respectively, $f:X \rightarrow Y$ be a bi-continuous map and $u,w \in D_1$. Then the following hold:
\begin{enumerate}
  \item $f^{-1}(u^*)=f^{-1}(Cl_1(u)^c)=f^{-1}(u)^*=Cl_1(f^{-1}(u))^c \leftrightarrow f^{-1}(Bd_1(u)) \subseteq Cl_1(f^{-1}(u))$. 
  \item $f^{-1}(u^+)=f^{-1}(Cl_2(u^c))=f^{-1}(u)^+=Cl_2(f^{-1}(u)^c) \leftrightarrow f^{-1}(Bd_2(u^c)) \subseteq Cl_2(f^{-1}(u^c))$.
\end{enumerate}
\end{lemma}
\begin{proof}
Recall $f^{-1}:D_1 \rightarrow B_1$ is known to be a bounded lattice homomorphism.\cite{guram}
\begin{enumerate}
  \item 
     \begin{enumerate}
       \item $\leftarrow f^{-1}(u^* \wedge u)=f^{-1}(u^*) \wedge f^{-1}(u)=f^{-1}(u)^* \wedge f^{-1}(u)=\emptyset$, hence $f^{-1}(u^*) \subseteq f^{-1}(u)^*$.\\ Now let $x \in f^{-1}(u)^* = Cl_1(f^{-1}(u))^c$ and note $f^{-1}(u^*)=f^{-1}(Cl_1(u)^c)=f^{-1}(Cl_1(u))^c$. So suppose $x \in f^{-1}(Cl_1(u))=f^{-1}(Bd_1(u)) \vee f^{-1}(u)$ and consider
	\begin{enumerate}
	  \item $x \in f^{-1}(u)$ contradicts $x \in  Cl_1(f^{-1}(u))^c$.
	  \item $x \in f^{-1}(Bd_1(u)) \rightarrow  x \in Cl_1(f^{-1}(u))$ by supposition, again yielding a contradiction. 
	\end{enumerate}
       \item $\rightarrow$ $f^{-1}(Cl_1(u))^c = Cl_1(f^{-1}(u))^c \leftrightarrow f^{-1}(Cl_1(u)) = Cl_1(f^{-1}(u))$
     \end{enumerate}
  \item
     \begin{enumerate}
       \item $\leftarrow f^{-1}(u^+ \vee u)=f^{-1}(u^+) \vee f^{-1}(u)=f^{-1}(u)^+ \vee f^{-1}(u)=X$, hence $f^{-1}(u)^+ \subseteq f^{-1}(u^+)$.\\ Now let $x \in f^{-1}(u^+) = f^{-1}(Cl_2(u^c))=f^{-1}(Bd_2(u^c)) \vee f^{-1}(u^c)$ and note $f^{-1}(u)^+=Cl_2(f^{-1}(u)^c)=Cl_2(f^{-1}(u^c))$.
       \begin{enumerate}
	\item $x \in  f^{-1}(u^c) \rightarrow x \in Cl_2(f^{-1}(u)^c)$.
	\item $x \in  f^{-1}(Bd_2(u^c)) \rightarrow x \in Cl_2(f^{-1}(u^c))$ by supposition.
       \end{enumerate}
     \item $\rightarrow  f^{-1}(Bd_2(u^c)) \subseteq f^{-1}(u^+) =f^{-1}(u)^+=Cl_2(f^{-1}(u^c))$
    \end{enumerate}
\end{enumerate}
\end{proof}
Now that we have compatibility with $^*$ and $^+$, showing $f^{-1}$ is a core regular double stone algebra homomorphism follows easily.
\begin{corollary}\label{inv2core}
Let f be as in Lemma \ref{inv2phom} and $k_d$ and $k_b$ be the core elements of $D_1$ and $B_1$, respectively, then $f^{-1}(k_d)=k_b$.
\end{corollary}
\begin{proof}
$f^{-1}(k_d)^*=f^{-1}(k_d^*)=\emptyset$ and $f^{-1}(k_d)^+=f^{-1}(k_d^+)=X$.
\end{proof}
We note that the above proof works in general, hence we see no need to require this as part of the definition of a core regular double Stone algebra homomorphism as was done in \cite{CRDSA}. We now have enough to show our main result:

\begin{main2}\label{finvcrsdahom}
Let $(X,t_1,t_2)$ and $(Y,\gamma_1,\gamma_2)$ be a bitopological space with CRDSA bases $B_1$, $B_2$ and $D_1$, $D_2$, respectively, $f:X \rightarrow Y$ be a bi-continuous map and $u,w \in D_1$. The $f^{-1}$ is a CRDSA homomorphism $\leftrightarrow f^{-1}(Bd_1(u)) \subseteq Cl_1(f^{-1}(u))$ and $f^{-1}(Bd_2(u^c)) \subseteq Cl_2(f^{-1}(u^c))$.
\end{main2}

These results are topologically indicative of just how ``nearly Boolean'' CRDSA are, these inverses must respect the appropriate conditions on the boundary of non-clopen elements of the bases of $t_1$ and $t_2$. Now we will draw on results from \textit{Bitopological duality for distributive lattices and Heyting algebras} to establish the dual equivalence we seek. We will rely heavily on section 5. \textit{Distributive lattices and pairwise Stone spaces} of \cite{guram}, where a dual equivalence between the categories of bounded distributive lattices and pariwise Stone spaces is established. We begin with the following definitions:

\begin{definition}
Let $h:L \rightarrow L'$ be a CRDSA homomorphism and $pf(L)$, $pf(L')$ denote the sets of prime filters of $L$, $L'$, respectively. Define $f_h:pf(L') \rightarrow pf(L)$ by $f_h(x)=h^{-1}(x)$.
\end{definition}
\begin{note}
In section 5. of \cite{guram} $f_h$ is shown to be a bi-continuous map and it is also shown that for $a \in L$, $f_h^{-1}(\Phi_+(a))=\Phi_+'(h(a))$.
\end{note}
\begin{lemma}
Let $B_+$ denote the base of $t_+$ on $pf(L)$. $f_h$ satisfies the conditions of the Main Theorem, namely for $b \in B_+$, $f_h^{-1}(Bd_1(b)) \subseteq Cl_1(f_h^{-1}(b))$ and $f_h^{-1}(Bd_2(b^c)) \subseteq f_h^{-1}(b^c))$.
\end{lemma}
\begin{proof}
We know from Corollary \ref{subgen1} that $\Phi_+$ is a CRDSA isomorphism so for any $b \in B_+$, $\exists! a \in L$ such that $b=\Phi_+(a)$. So, $f_h^{-1}(b^*)=f_h^{-1}(\Phi_+(a^*))=\Phi_+'(h(a^*))=\Phi_+'(h(a)^*)=\Phi_+'(h(a))^*=f_h^{-1}(\Phi_+(a))^*=f_h^{-1}(b)^*$ and similarly for $^+$. Now the result follows directly from Lemma \ref{inv2phom} and Corollary \ref{inv2core}.
\end{proof}

\begin{note}
In section 5. of \cite{guram} in considering one composition of contravariant functors necessary for dual equivalence it is shown that for a bounded distributive lattice L, we have $L_*$$^* = \Phi_+[L]$ and we already have that $\Phi_+$ is a CRDSA isomorphism from Corollary \ref{subgen1}. We now consider the other necessary composition of functors.
\end{note}

\begin{definition}\label{Psi}
For a pairwise Stone space $(X,t_1,t_2)$, let $\Psi:X \rightarrow X{^*}_*$ be given by $\Psi(x)=\{u \in X^* | x \in u\}$.
\end{definition}

\begin{note}
Again in section 5. of \cite{guram} the situation $(X,t_1,t_2) \rightarrow (B_1,\wedge,\vee,\emptyset,X) \rightarrow X{^*}_*$ is described and the latter map utilized is $\Phi_+$.  It is shown that for $u \in B_1$, $\Psi^{-1}(\Phi_+(u))=u \in B_1$ and $\Psi$ is a bi-homeomorphism.
\end{note}
 
We need a bit more than a bi-homeomorphism, namely we need $\Psi$ to be such that it satisfies the extra conditions of Lemma \ref{inv2phom}. We establish this in the following Lemma:

\begin{lemma}
$\Psi^{-1}$ saisfies the conditions of the Main Theorem, namely for $a \in B_+$ where $B_+$ is the corresponding base of $X{^*}_*$, $\Psi^{-1}(Bd_1(a)) \subseteq Cl_1(\Psi^{-1}(a))$ and $\Psi^{-1}(Bd_2(a^c)) \subseteq Cl_2(\Psi^{-1}(a^c))$. 
\end{lemma}
\begin{proof}
By Corollary \ref{subgen1} $\Phi_+$ is a CRDSA isomorphism between $B_1$ and $B_+$, so $\exists! u \in B_1$ such that $\Phi_+(u)=a$, $\Phi_+(u^*)=\Phi_+(u)^*=a^*$ and $\Phi_+(u^+)=\Phi_+(u)^+=a^+$. So we have that $\Psi^{-1}(a)=\Psi^{-1}(\Phi_+(u))=u$, $\Psi^{-1}(a^*)=\Psi^{-1}(\Phi_+(u^*))=u^*$ and $\Psi^{-1}(a^+)=\Psi^{-1}(\Phi_+(u^+))=u^+$. Now the result follows directly from Lemma \ref{inv2phom} and Corollary \ref{inv2core}.
\end{proof}

Hence we can now claim the following:

\begin{mainc1}\label{crsdadual}
The category of core regular double Stone algebras is dually equivalent to the category of core regular double pairwise Stone spaces.
\end{mainc1}
\begin{proof}
We are using essentially the same functors as in \cite{guram}, we are simply restricting it to CRDSA. Hence this follows directly from the results in section 5. \textit{Distributive lattices and pairwise Stone spaces} of \cite{guram} and the results in this note.
\end{proof}

\section{Conclusion}
In \cite{CRDSA} many useful results regarding the center of a core regular double Stone algebra, CRDSA, that begin to indicate the \textit{nearly Boolean} nature of CRDSA which we focused on here. We started with a model of network security where individual nodes are considered to be in one of 3 states. Let J be any non-empty set of network nodes, not necessarily finite. We define the node set bounded distributive lattice through the pairwise disjoint subsets of J with the well known binary operations of ternary set partitions and note J = {1} is our minimal case. We then showed the resultant bounded distributive lattice is isomorphic to $C_3^J$ where $C_3$ is the 3 element chain CRDSA. We derived that \textbf{every CRDSA is a subdirect product of} $\mathbf{C_3}$, similarly as for Boolean algebras and $C_2$. We used these results along with a few known results to show a main result, namely \textbf{every Boolean algebra is the center of some core regular double Stone algebra, CRDSA}. We then used that result to characterize all subalgebras of a finite core regular double Stone Algebras, namely \textbf{a CRDSA A is a subalgebra of} $\mathbf{C_3^J}$ \textbf{for some finite J} $\mathbf{\leftrightarrow A \cong C_3^K}$ \textbf{for some} $\mathbf{K \leq J}$.\\ \indent Next we showed that $\mathbf{C_3}$ \textbf{is primal} which implies that \textbf{the variety generated by} $\mathbf{C_3}$ \textbf{is dually equivalent to the category of Stone spaces and hence the category of Boolean algebras.} In some sense this is a last step towards our goal of establishing CRDSA as \textit{nearly Boolean}, but leaves us a bit dissatisfied as to our understanding of CRDSA in the dual topological category. Hence we continued by establishing a duality between the category of CRDSA and specifically crafted bi-topological spaces that enables better understanding of the "nearly boolean" nature of CRDSA in the dual category.  More succinctly, duality through  refinement of a pre-established duality of pairwise Stone spaces and bounded distributive lattices \cite{guram}, and we used \cite{guram} as our primary source for bi-topological spaces as well.\\ \indent Towards this end we first established \textbf{necessary and sufficient conditions on a pairwise zero-dimensional space such that it will have a core regular double Stone algebra base}. For the purposes of this note only we called any pairwise zero-dimensional space $(X,t_1,t_2)$ with a core regular double Stone algebra bases $(B_1,\vee,\wedge,\emptyset,X)$ and $(B_2,\vee,\wedge,\emptyset,X)$ a core regular double pairwise zero-dimensional space. We note that these conditions are indicative of the \textit{nearly Boolean} CRDSA are.(Recall given any topological space X, the collection of subsets of X that are clopen (both closed and open) is a Boolean algebra.) For example, if $u \in B_1$ is not $t_1$ clopen/complemented then $Cl_1(u)\in B_1$ and is $t_1$ clopen/complemented. \\ \indent Next we showed that \textbf{the pairwise Stone space derived from a CRDSA L has a base that is a CRDSA isomorphic to L}. For the purposes of this note only we called any such pairwise Stone space a core regular double pairwise Stone space. Then we establish \textbf{necessary and sufficient conditions for a bi-continuous map to have an inverse that is a CRDSA homomorphism}. These results are topologically indicative of just how ``nearly Boolean'' CRDSA are, these inverses must respect the appropriate conditions on the boundary of non-clopen elements of the bases of $t_1$ and $t_2$. \\ \indent From that result we validate the claim that \textbf{the category of core regular double Stone algebras is dually equivalent to the category of core regular double pairwise Stone spaces}. We note that the conditions for this duality should be able to be easily relaxed to yield a duality for a less rigid subclass of bounded distributive lattices than CRDSA, bounded distributive pseudo-complemented lattices for example.

\end{document}